\documentclass{amsart}

\usepackage{color}

\renewcommand\phi{\varphi}

\newcommand\A{\mathbb A}
\renewcommand\P{\mathbb P}

\DeclareMathOperator\const{{const}}
\DeclareMathOperator\res{{res}}
\DeclareMathOperator\gr{{gr}}
\DeclareMathOperator{\Id}{Id}
\DeclareMathOperator{\rk}{rk}
\DeclareMathOperator{\Ker}{Ker}
\DeclareMathOperator{\mult}{mult}

\newcommand{\E}{\mathrm{e}}

\newcounter{noindnum}[subsection]
\renewcommand{\thenoindnum}{\alph{noindnum}}
\newcommand{\noindstep}{\refstepcounter{noindnum}{\rm(}\thenoindnum\/{\rm)} }
\newcommand{\stepzero}{\setcounter{noindnum}{0}
}
\theoremstyle{plain}
\newtheorem{theorem}{Theorem}
\newtheorem*{theorem*}{Theorem}
\newtheorem{proposition}{Proposition}[section]
\newtheorem{lemma}[proposition]{Lemma}
\newtheorem{corollary}[proposition]{Corollary}

\theoremstyle{definition}
\newtheorem{definition}[proposition]{Definition}

\theoremstyle{remark}
\newtheorem{remark}[proposition]{Remark}
\newtheorem{remarks}[proposition]{Remarks}

\newcommand\cA{{\mathcal{A}}}
\newcommand\cO{{\mathcal{O}}}
\newcommand\cV{{\mathcal{V}}}

\newcommand\Z{\mathbb Z}
\newcommand\R{\mathbb R}
\newcommand\C{\mathbb C}

\newcommand{\Higgs}{\mathcal{H}iggs}

\begin{document}

\title[VHS for Hypergeometric differential operators]{Variations of Hodge structures for hypergeometric differential operators and parabolic Higgs bundles}

\begin{abstract}
Consider the holomorphic bundle with connection on $\P^1-\{0,1,\infty\}$ corresponding to the regular hypergeometric differential operator
\[
    \prod_{j=1}^h(D-\alpha_j)-z\prod_{j=1}^h(D-\beta_j),\qquad D=z\frac{d}{dz}.
\]
If the numbers $\alpha_i$ and $\beta_j$ are real and for all $i$ and $j$ the number $\alpha_i-\beta_j$ is not integer, then the bundle with connection is known to underlie a complex polarizable variation of Hodge structures. We calculate some Hodge invariants for this variation, in particular, the Hodge numbers. From this we derive a conjecture of Alessio Corti and Vasily Golyshev. We also use non-abelian Hodge theory to interpret our theorem as a statement about parabolic Higgs bundles.
\end{abstract}

\author{Roman Fedorov}
\email{rmfedorov@gmail.com}
\address{Mathematics Department, 138 Cardwell Hall, Manhattan, KS 66506, USA}

\vskip5cm
\maketitle

\section{Introduction}
Fix two sequences of real numbers
\begin{equation}\label{eq:alphabeta}
\begin{split}
    0\le\alpha_1\le\ldots\le\alpha_h<1\\
    0\le\beta_1\le\ldots\le\beta_h<1.
\end{split}
\end{equation}
Assume that for all $i,j$ we have $\alpha_i\ne\beta_j$. The \emph{regular hypergeometric differential operator\/} is the following complex differential operator
\begin{equation}\label{eq:diffop}
    \prod_j(D-\alpha_j)-z\prod_j(D-\beta_j),\qquad D=z\frac{d}{dz}.
\end{equation}
This differential operator has been studied in details in~\cite{BeukersHeckman} as the differential operator annihilating the generalized hypergeometric function $_nF_{n-1}$.

Expanding, we see that the leading term of this operator is $z^h(1-z)(d/dz)^h$, so the operator gives rise to a connection $\nabla$ on the trivial rank $h$ holomorphic vector bundle $V$ over $\P^1-\{0,1,\infty\}$. The bundle with connection $(V,\nabla)$ is known to be \emph{irreducible\/} and \emph{physically rigid}. The latter means that a bundle with connection $(V',\nabla')$ is isomorphic to $(V,\nabla)$, provided its monodromy at each of the singular points is conjugate to that of $(V,\nabla)$ (see Section~\ref{sect:MonIrrRig} below). Rigid bundles with connections form, in a sense, the simplest class of bundles with connections. In particular, the Katz algorithm allows to reduce such to a rigid bundle with connection of a smaller rank by tensoring with a line bundle followed by a middle convolution (see~\cite[Sect.~5]{Simpson:MiddleConv} and also~\cite{KatzRigid,ArinkinRigid,DettweilerSabbah,DettweilerReiter,BlochEsnault:Fourier}).

It was conjectured by Alessio Corti and Vasily Golyshev~\cite{CortiGolyshev} that $(V,\nabla)$ underlies a \emph{real polarizable variation of Hodge structures}, provided that for $m=1,\ldots,h$ we have
\[
    \alpha_m+\alpha_{h+1-m}\in\Z,\quad \beta_m+\beta_{h+1-m}\in\Z.
\]
More importantly, they gave conjectural formulas for the Hodge numbers. We prove these conjectures below, see Theorem~\ref{th:real}.

To prove the conjectures, we note first that by~\cite[Cor.~8.1]{SimpsonHarmonicNoncompact} every rigid irreducible bundle with connection underlies a \emph{complex polarizable variation of Hodge structures} (this structure is essentially unique by~\cite[Prop.~1.13(i)]{DeligneFinitude}). This notion will be recalled in Section~\ref{sect:CPVHS}. Our first main result is the calculation of the corresponding Hodge numbers for $(V,\nabla)$.

\begin{theorem*}
Set $\rho(k):=\#\{j:\alpha_j<\beta_k\}-k$. Then we have up to a shift
\[
    h^p=\#\rho^{-1}(p)=\#\{k:1\le k\le h,\rho(k)=p\}.
\]
\end{theorem*}

This is the content of Theorem~\ref{th:GolyshevCorti} below. Our proof of this theorem is based on the technique of middle convolution and on the paper~\cite{DettweilerSabbah}, where it is explained how Hodge data changes under middle convolution. In fact, we calculate not only Hodge numbers but also the numerical invariants of the limits of Hodge structures as $z$ tends to 0 or $\infty$; see Theorem~\ref{th:nu}.

Finally, according to the results of~\cite{SimpsonHarmonicNoncompact}, complex variations of Hodge structures correspond under non-abelian Hodge theory to certain decompositions of Higgs bundles. We translate our Theorem~\ref{th:GolyshevCorti} into a statement about Higgs bundles: the numbers $\alpha_i$ and $\beta_i$ give a stability condition on the moduli of Higgs bundles and we describe the unique stable Higgs bundle, see Theorem~\ref{th:Higgs}.

\section{Main results}\label{sect:main}
We keep the notation of the introduction. Thus $\alpha_i$ and $\beta_i$ are real numbers satisfying~\eqref{eq:alphabeta} and such that $\alpha_i\ne\beta_j$ for all $i,j\in\{1,\ldots,h\}$. Next, $(V,\nabla)$ denotes the holomorphic bundle with connection over $\P^1_\C-\{0,1,\infty\}$ corresponding to the regular hypergeometric operator~\eqref{eq:diffop}. Thus, $\nabla:V\to V\otimes\Omega^1$ is a $\C$-linear map.

We denote by $\mult(\alpha_i)$ the multiplicity of $\alpha_i$, that is, the number of elements in the set $\{j:\alpha_j=\alpha_i\}$. We fix a square root of $-1$ and set $\E(x):=\exp(2\pi\sqrt{-1}x)$. We denote by $\{\alpha\}$ the fractional part of $\alpha$, that is, $\{\alpha\}$ is uniquely defined by the conditions that $0\le\{\alpha\}<1$ and $\alpha-\{\alpha\}\in\Z$.

\subsection{Monodromy, irreducibility, and rigidity}\label{sect:MonIrrRig} We would like to recall some results from~\cite{BeukersHeckman}. First of all, we recall the local monodromy of $(V,\nabla)$. Let $A$ be a regular linear operator acting on an $h$-dimensional space such that the list of eigenvalues of $A$ with multiplicities is $\alpha_1,\ldots,\alpha_h$. In other words, for each $i$, $A$ has a unique Jordan block with eigenvalue $\alpha_i$ of the size $\mult(\alpha_i)$. Similarly, let~$B$ be a regular operator of acting on an $h$-dimensional vector space whose list of eigenvalues with multiplicities is $-\beta_1,\ldots,-\beta_h$. Recall that the local monodromy of a bundle with connection around a singular point is defined up to conjugation.

\begin{proposition}\label{pr:localmonodromy}

\stepzero\noindstep The local monodromy of $(V,\nabla)$ at $z=0$ is $\E(A)$.

\noindstep The local monodromy of $(V,\nabla)$ at $z=\infty$ is $\E(B)$.

\noindstep The local monodromy of $(V,\nabla)$ at $z=1$ is a pseudo-reflection (that is, the sum of the identity operator and a rank one operator).
\end{proposition}
\begin{proof}
Combine Prop.~3.2 and Theorem~3.5 of~\cite{BeukersHeckman}.
\end{proof}

We note that the conjugacy class of the monodromy at $z=1$ is uniquely determined by the proposition. Indeed, the conjugacy class of a pseudo-reflection is determined by its determinant. However, this determinant is equal to $(\det(\E(A)\E(B)))^{-1}$.

Next, we have
\begin{proposition}\label{pr:irreducible}
The bundle with connection $(V,\nabla)$ is irreducible.
\end{proposition}
\begin{proof}
Combine Prop.~3.2 and Prop.~3.3 of~\cite{BeukersHeckman}.
\end{proof}

\begin{remark}
The condition $\alpha_i\ne\beta_j$ is necessary for the bundle with connection to be irreducible, see~\cite[Prop.~2.7]{BeukersHeckman}.
On the other hand, the conditions $0\le\alpha_i<1$ and $0\le\beta_i<1$ are unimportant: if we add an integer to any of $\alpha_i$ or $\beta_j$, we get an isomorphic bundle with connection, see~\cite[Cor.~2.6]{BeukersHeckman}.
\end{remark}

The bundle with connection $(V,\nabla)$ is \emph{physically rigid}. Precisely, this means that the following holds.
\begin{proposition}\label{pr:rigidity}
Let $(V',\nabla')$ be a bundle with connection over $\P^1_\C-\{0,1,\infty\}$ such that for the local monodromy of $(V',\nabla')$ all the statements of Proposition~\ref{pr:localmonodromy} hold. Then $(V',\nabla')$ is isomorphic to $(V,\nabla)$.
\end{proposition}
\begin{proof}
This is Theorem~3.5 of~\cite{BeukersHeckman}.
\end{proof}

\subsection{Complex polarizable variations of Hodge structures}\label{sect:CPVHS}
Let $X$ be a non-singular complex manifold. (It will be the complement of three points in $\P^1_\C$ in our applications.) Let $V$ be a holomorphic vector bundle with a flat holomorphic connection $\nabla:V\to V\otimes\Omega^1_X$. We can extend $\nabla$ to a flat smooth connection
\begin{equation}\label{eq:dnabla}
    D_\nabla:\cA^0(V)\to\cA^{1,0}(V)\oplus\cA^{0,1}(V),
\end{equation}
where $\cA^{i,j}(V)$ is the sheaf of smooth complex $V$-valued differential forms of type $(i,j)$. Thus $D_\nabla=D_\nabla'+D_\nabla''$, where $D_\nabla':\cA^0(V)\to\cA^{1,0}(V)$ is induced by $\nabla$ and $D_\nabla''=\bar\partial:\cA^0(V)\to\cA^{0,1}(V)$ is induced by the holomorphic structure on $V$.

A \emph{complex variation of Hodge structures\/} is a decomposition of $V$ into a direct sum of \emph{smooth} complex vector bundles
\begin{equation}\label{eq:HodgeDecomp}
    V=\bigoplus_{p\in\Z}H^p
\end{equation}
such that
\[
D_\nabla'(\cA^0(H^p))\subset\cA^{1,0}(H^p\oplus H^{p-1}),\qquad
D_\nabla''(\cA^0(H^p))\subset\cA^{0,1}(H^p\oplus H^{p+1}).
\]
Denote by $h^p$ the \emph{Hodge number\/} $\rk H^p$. Thus we have $\sum_p h^p=\rk V$.

This variation of Hodge structures is \emph{polarizable\/}, if there is a $D_\nabla$-flat hermitian form $Q$ on $V$ such that the decomposition~\eqref{eq:HodgeDecomp} is $Q$-orthogonal and the restriction of $(-1)^pQ$ to $H^p$ is positive definite. We abbreviate a complex polarizable variation of Hodge structures as CPVHS, and denote it by $(V,\nabla,H^\bullet)$. We also say that the decomposition $H^\bullet$ is a CPVHS on $(V,\nabla)$.

\begin{remark}\label{rm:shift}
Let $(V,\nabla,H^\bullet)$ be a CPVHS. Then for any $c\in\Z$ we have a shifted CPVHS
$(V,\nabla,H^{\bullet+c})$.
\end{remark}

Now let $X=\P_\C^1-\{0,1,\infty\}$ and let the hypergeometric bundle with connection $(V,\nabla)$ be as in Section~\ref{sect:MonIrrRig}.

\begin{proposition}\label{pr:ExistsUniqueCPVHS}
    The bundle with connection $(V,\nabla)$ is a part of a unique up to a~shift CPVHS.
\end{proposition}
\begin{proof}
    Since the bundle with connection is rigid and irreducible (Propositions~\ref{pr:rigidity} and~\ref{pr:irreducible}), the existence of CPVHS follows from~\cite[Cor.~8.1]{SimpsonHarmonicNoncompact}, see also~\cite[Thm.~2.4.1]{DettweilerSabbah}. Alternatively, it follows by induction on the rank $h$ from Lemma~\ref{lm:MC} below and~\cite[Prop.~3.1.1]{DettweilerSabbah}. The uniqueness follows from~\cite[Prop.~1.13(i)]{DeligneFinitude}.
\end{proof}
Let $h^p$ be the corresponding Hodge numbers (defined up to a shift $h^p\mapsto h^{p+\const}$, see Remark~\ref{rm:shift}). Our first main result is
\begin{theorem}\label{th:GolyshevCorti}
Set $\rho(k):=\#\{j:\alpha_j<\beta_k\}-k$. Then we have up to a shift
\[
    h^p=\#\rho^{-1}(p)=\#\{k:1\le k\le h,\rho(k)=p\}.
\]
\end{theorem}
This theorem follows immediately from Theorem~\ref{th:nu}\eqref{thitem:HodgeNumbers} below.
\begin{remark}
It seems plausible that this theorem can be proved using results of Terasoma~\cite{TerasomaRadon,TerasomaHodgeAndTate}. On the other hand, our proof is very natural from the perspective of the theory of bundles with connections on curves.
\end{remark}

\begin{corollary}
The bundle with connection $(V,\nabla)$ admits a hermitian flat metric of signature $\sum_{k=1}^h(-1)^{\rho(k)}$.
\end{corollary}
\begin{proof}
By definition of CPVHS, $(V,\nabla)$ has a flat hermitian metric of signature $\sum_p(-1)^ph^p$. Now apply the theorem.
\end{proof}
The automorphism of $\P_\C^1$ given by $\phi(z)=1/z$ takes $(V,\nabla)$ to a similar bundle with connection corresponding to $\alpha_m=\{-\beta_m\}$, $\beta_m=\{-\alpha_m\}$. Using this, one can check that the previous corollary is equivalent to~\cite[Thm.~4.5]{BeukersHeckman}.

\begin{corollary}
The following are equivalent\\
\stepzero\noindstep\label{unitaryA} The numbers $\alpha_j$ and $\beta_j$ interlace, that is, either
\[
    \alpha_1<\beta_1<\alpha_2<\beta_2<\ldots<\alpha_h<\beta_h\;\text{ or }\;
    \beta_1<\alpha_1<\beta_2<\alpha_2<\ldots<\beta_h<\alpha_h.
\]
\noindstep\label{unitaryB} The trivial complex variation of Hodge structures $H^0=V$, $H^p=0$ if $p\ne0$ is polarizable.\\
\noindstep\label{unitaryC} $(V,\nabla)$ admits a positive-definite flat hermitian metric. That is, the monodromy representation is unitary.
\end{corollary}
\begin{proof}
The implication \eqref{unitaryA}$\Rightarrow$\eqref{unitaryB} follows from the theorem immediately. Conversely, if~\eqref{unitaryB} is satisfied, then, using uniqueness of CPVHS up to a~shift and the theorem, we see that $\rho(k)$ is constant. Since $\rho(h)\le0$ and $\rho(1)\ge-1$, this constant is either~0 or $-1$. One easily derives that $\alpha_j$ and $\beta_j$ interlace.

The equivalence of~\eqref{unitaryB} and~\eqref{unitaryC} follows from the definition of CPVHS.
\end{proof}
Note that the equivalence of~~\eqref{unitaryA} and~\eqref{unitaryC} above is the content of~\cite[Cor.~4.7]{BeukersHeckman}.

\subsection{Real variations of Hodge structures and a conjecture of Corti and Golyshev}
Assume now that for $m=1,\ldots,h$ we have
\begin{equation}\label{eq:real}
\alpha_{h+1-m}+\alpha_m\in\Z,\qquad\beta_{h+1-m}+\beta_m\in\Z.
\end{equation}
In this case $(V,\nabla)$ underlies a real variation of Hodge structures as we show below.

Let us recall the definition. Let $\cV_\R$ be a local system of real vector spaces on a~smooth complex manifold $X$. Let $\cV_\R\otimes_\R\C$ be its complexificaiton, and let $(V,\nabla)$ be the holomorphic bundle with connection corresponding to this local system under the Riemann--Hilbert correspondence. Let $k$ be an integer. Let $H^{\bullet,\bullet}$ be a~decomposition into the sum of smooth complex vector bundles
\begin{equation}\label{eq:realHodge}
    V=\bigoplus_{p+q=k}H^{p,q}
\end{equation}
such that for all $p,q$ we have $\overline{H^{p,q}}=H^{q,p}$ and
\[
    D_\nabla(\cA^0(H^{p,q}))\subset\cA^{1,0}(H^{p,q}\oplus H^{p-1,q+1})\oplus\cA^{0,1}(H^{p,q}\oplus H^{p+1,q-1}),
\]
where, as before, $D_\nabla$ is the smooth connection corresponding to $\nabla$. In this case we say that $(\cV_\R,H^{\bullet,\bullet})$ is a \emph{real variation of Hodge structures of weight $k$}.

Let $Q_\R$ be a bilinear form of parity $(-1)^k$ on $\cV_\R$, denote by $Q$ the induced bilinear form on $V$. We say that $Q_\R$ is a \emph{polarization\/} of $(\cV_\R,H^{\bullet,\bullet})$ if the Hodge decomposition~\eqref{eq:realHodge} is $Q$-orthogonal and for $x\in H^{p,q}$, $x\ne0$ we have $\left(\sqrt{-1}\right)^{p-q}Q(x,\bar x)>0$.

A real variation of Hodge structures $(\cV_\R,H^{\bullet,\bullet})$ is called polarizable, if a polarization exists. Note that in this case, setting $H^p:=H^{p,k-p}$, we get a CPVHS $(V,\nabla,H^\bullet)$. ($V$ and $\nabla$ are as above.)

Let $\rho$ be as in Theorem~\ref{th:GolyshevCorti}, let $(V,\nabla)$ be the hypergeometric local system from Section~\ref{sect:MonIrrRig}. The following has been conjectured by Alessio Corti and Vasily Golyshev; see Conjecture~1.4 of~\cite{CortiGolyshev}.
\begin{theorem}\label{th:real}
    Let $p_+=\max\rho(k)$, $p_-=\min\rho(k)$. Then there is a real variation of Hodge structures $(\cV_\R,H^{\bullet,\bullet})$ of weight $p_+-p_-$ such that the bundle with connection corresponding to $\cV_\R\otimes_\R\C$ is isomorphic to $(V,\nabla)$, and
    \[
        \rk\left(H^{k-p_-,-k+p_+}\right)=\#\rho^{-1}(k).
    \]
\end{theorem}
This theorem will be proved in Section~\ref{sect:real}.
\begin{remark}
If $(\cV_\R,H^{\bullet,\bullet})$ is a real polarizable variation of Hodge structures of weight $k$, then setting ${'H}^{p,q}=H^{p+c,q+c}$ we get a real polarizable variation of Hodge structures of weight $k+2c$ with the same $\cV_\R$. Using this operation, one can always assume that $H^{0,k}\ne0$, $H^{k,0}\ne0$, $H^{p,q}=0$ if $p<0$ or $q<0$. This is exactly the normalization of the above theorem.
\end{remark}

\subsection{Local Hodge invariants}
Let $(V,\nabla)$ be as in Section~\ref{sect:MonIrrRig}. Recall from Proposition~\ref{pr:ExistsUniqueCPVHS} that $(V,\nabla)$ can be extended to a CPVHS on $\P^1_\C-\{0,1,\infty\}$. We are going to describe the limits of this CPVHS at $z=0$ and $z=\infty$. Let us give the corresponding definitions in the general case. We follow~\cite[\S2.1]{DettweilerSabbah}.

So, let now $(V,\nabla)$ be an arbitrary holomorphic bundle with connection on a~smooth complex curve~$X$. Let $\bar X\supset X$ be a curve such that $\bar X-X=\{x\}$ is a single point. Assume that the eigenvalues of the monodromy of $(V,\nabla)$ around~$x$ lie on the unit circle.
The following is known as Deligne's canonical extension.
\begin{lemma}\label{lm:DeligneExt}
    For all $\alpha\in\R$ there is a unique extension $(V^\alpha,\nabla^\alpha)$ of $(V,\nabla)$ to $\bar X$ such that
    \begin{itemize}
        \item $V^\alpha$ is a holomorphic vector bundle on $\bar X$;
        \item $\nabla^\alpha:V^\alpha\to V^\alpha\otimes\Omega_{\bar X}(x)$ is a connection with a first order pole at $x$;
        \item The eigenvalues of the residue\footnote{Our convention is that $\res(d-A/z)=\res(A)$.} of $\nabla^\alpha$ at $x$ belong to $[\alpha,\alpha+1)$.
    \end{itemize}
\end{lemma}
\begin{proof}
We may assume that $\bar X$  is a disc and $x$ is its center. In this case the statement is a particular case of~\cite[Ch.~II, Prop.~5.4]{DeligneSingulierRegulier}.
\end{proof}
Replacing $[\alpha,\alpha+1)$ by $(\alpha,\alpha+1]$, we get a definition of a similar extension $(V^{>\alpha},\nabla^{>\alpha})$ such that the eigenvalues of $\res\nabla^{>\alpha}$ are in $(\alpha,\alpha+1]$. Let $\jmath:X\to\bar X$ be the open embedding, consider the pushforward $\jmath_*V$. Identifying $V^\alpha$ with its sheaf of sections, we can view $V^\alpha$ as a subsheaf of $\jmath_*V$. It is not difficult to see that $V^\alpha\supset V^\beta$ if $\alpha<\beta$, $V^{>\alpha}=\cup_{\beta>\alpha}V^\beta$, and the quotients $V^\alpha/V^\beta$, $V^\alpha/V^{>\alpha}$ are set theoretically supported at $x$. Let
\[
    \psi^\alpha(V)=\psi_x^\alpha(V,\nabla):=\Gamma(\bar X,V^\alpha/V^{>\alpha}).
\]
be the space of moderate nearby cycles. It is equipped with a nilpotent endomorphism $N_\alpha$ induced by $z\nabla_{\frac d{dz}}-\alpha$, where $z$ is any coordinate near $x$ such that $z(x)=0$.

\begin{remark}\label{rm:alphaPlus1}
We have $V^{\alpha+1}=V^\alpha(-x)$, so $V^\alpha/V^{\alpha+1}=V^\alpha_x$ is the fiber of $V^\alpha$ at~$x$. It is not difficult to see that for $\beta\in[\alpha,\alpha+1)$ we can identify $\psi^\beta(V)$ with the generalized eigenspace of the residue of $\nabla^\alpha$ corresponding to the eigenvalue $\beta$.
\end{remark}

Next, we have a general statement (see~\cite[Lemma~6.4]{SchmidSingularities}).
\begin{lemma}
Let $L$ be a finite-dimensional vector space, let $N$ be a nilpotent endomorphism of $L$. Then

\stepzero\noindstep There is a unique increasing filtration $W^\bullet L$ indexed by integers such that for all $k$ we have $N(W^kL)\subset W^{k-2}L$ and for all $k>0$ $N^k$ induces an isomorphism $\gr_W^kL\simeq\gr_W^{-k}L$.

\noindstep For $l\ge0$ set
        \[
            P^lL:=(\Ker\bar N^{l+1})\cap\gr_W^lL,
        \]
     where $\bar N$ is the endomorphism of $\gr_WL$ induced by $N$.
     Then for $k\in\Z$ we have a~Lefschetz decomposition
        \[
            \gr_W^k L=\bigoplus_{l\ge0,}\bar N^lP^{k+2l}L.
        \]
\end{lemma}

Applying this to $\psi^\alpha(V)$ and $N_\alpha$, we get a filtration $W^k\psi^\alpha(V)$ and the primitive subspaces $P^l\psi^\alpha(V)\subset\gr^l_W\psi^\alpha(V)$ for $l\ge0$.

Now we are in position to define the local Hodge invariants. Let $(V,\nabla,H^\bullet)$ be a~CPVHS on $X=\bar X-x$. Note that $F^pV:=\bigoplus_{q\ge p}H^q$ is preserved by $D''$, so $F^pV$ is a holomorphic subbundle of $V$. We get a decreasing filtration $F^\bullet V$ on~$V$. This filtration uniquely extends to a filtration $F^\bullet{}V^\alpha$ on $V^\alpha$ by vector subbundles for $\alpha\in\R$ because partial flag varieties are compact. Moreover, it is easy to see that
\[
    F^pV^\alpha=\jmath_*F^pV\cap V^\alpha.
\]
Similarly, we get a filtration $F^\bullet{}V^{>\alpha}$ on $V^{>\alpha}$.

Setting $F^p\psi^\alpha(V):=F^pV^\alpha/F^pV^{>\alpha}$, we get filtrations on the spaces of moderate nearby cycles. These filtrations induce filtrations on $gr_W^l\psi^\alpha(V)$ and, in turn, on $P^l\psi^\alpha(V)$.

\begin{definition}
The \emph{local Hodge invariants\/} of $(V,\nabla,H^\bullet)$ at $x$ are
\[
    \nu_{\alpha,l}^p:=\dim Gr_F^pP^l\psi^\alpha(V)=\dim F^pP^l\psi^\alpha(V)-\dim F^{p+1}P^l\psi^\alpha(V),
\]
where $l\in\Z_{\ge0}$, $p\in\Z$, $\alpha\in[0,1)$.
\end{definition}

\begin{remarks}
\stepzero\noindstep
The above formulas define the invariants $\nu_\alpha^{p,l}$ for all $\alpha\in\R$ and it is not difficult to see that they only depend on $\alpha\bmod\Z$ (cf.~Remark~\ref{rm:alphaPlus1}). However, it will be convenient for us to have them defined only if $\alpha\in[0,1)$.

\noindstep
What we denote by $\nu_{\alpha,l}^p$ is denoted by $\nu_{\exp(2\pi\sqrt{-1}\alpha),l}^p$ in~\cite{DettweilerSabbah} (cf.~the previous remark). The ``additive'' notation is more convenient for us, we hope it will not lead to confusion.

\noindstep
It is easy to see that $\nu_{\alpha,l}^p=0$ unless $\E(\alpha)$ is an eigenvalue of the monodromy.
\end{remarks}

Later, we will also need similar `vanishing cycles' invariants
\begin{equation}\label{eq:munu}
    \mu_{\alpha,l}^p:=
    \begin{cases}
        \nu_{\alpha,l}^p\text{ if }\alpha\ne0\\
        \nu_{\alpha,l+1}^p\text{ if }\alpha=0.
    \end{cases}
\end{equation}

Now let us return to CPVHS from Proposition~\ref{pr:ExistsUniqueCPVHS} (recall that it is unique up to a shift). So, let $(V,\nabla)$ be the bundle with connection on $\P_\C^1-\{0,1,\infty\}$ constructed in the beginning of Section~\ref{sect:main}.

\begin{theorem}\label{th:nu}
There is a CPVHS  $(V,\nabla)$ such that\\
\stepzero\noindstep\label{thitem:0}
The local Hodge invariants at $z=0$ are given by
\begin{equation*}
\nu_{\alpha_m,l}^p=
\begin{cases}
    1\text{ if }p=\#\{i:\alpha_i\le\alpha_m\}-\#\{i:\beta_i\le\alpha_m\},\quad l=\mult(\alpha_m)-1\\
    0\text{ otherwise.}
\end{cases}
\end{equation*}
\noindstep\label{thitem:2} The local Hodge invariants at $z=\infty$ are given by
\begin{equation*}
\nu_{-\beta_m,l}^p=
\begin{cases}
    1\text{ if }p=\#\{i:\alpha_i<\beta_m\}-\#\{i:\beta_i<\beta_m\},\quad l=\mult(\beta_m)-1\\
    0\text{ otherwise.}
\end{cases}
\end{equation*}
\noindstep\label{thitem:HodgeNumbers}
The Hodge numbers are given by
\[
    h^p=\#\rho^{-1}(p-1).
\]
\end{theorem}

This theorem will be proved in Section~\ref{sect:proof}.

\subsection{Stability conditions on the moduli of Higgs bundles}
The results of Simpson~\cite{SimpsonHarmonicNoncompact} allow us to translate Theorem~1 into a statement about Higgs bundles as we presently explain. For a vector bundle $E$ we denote its fiber at a~point $z$ by $E_z$.

Let $\Higgs_{h,\delta}$ denote the moduli stack of collections
\begin{equation}\label{eq:FullFlags}
    (E,E_0^0,\ldots,E_0^h,E_\infty^0,\ldots,E_\infty^h,E_1^1,\Phi),
\end{equation}
where
\begin{itemize}
\item $E$ is a rank $h$ degree $\delta$ vector bundle over $\P_\C^1$;
\item $0=E_0^0\subset E_0^1\subset\ldots\subset E_0^h=E_0$ is a full flag in the fiber of $E$ at $z=0$;
\item $0=E_\infty^0\subset E_\infty^1\subset\ldots\subset E_\infty^h=E_\infty$ is a full flag in the fiber of $E$ at $z=\infty$;
\item $E_1^1\subset E_1$ is a 1-dimensional subspace in the fiber of $E$ at $z=1$;
\item $\Phi:E\to E\otimes\Omega_{\P_\C^1}(0+1+\infty)$ is a Higgs field possibly with singularities at 0, 1, and $\infty$;
\item for $z=0$ or $\infty$ the residue of $\Phi$ at $z$ is a nilpotent transformation compatible with the flag: $(\res_z\Phi)(E_z^i)\subset E_z^{i-1}$;
\item $(\res_1\Phi)(E_1)\subset E_1^1$, $(\res_1\Phi)(E_1^1)=0$.
\end{itemize}
We will abuse notation by writing $(E,E_0^i,E_\infty^i,E_1^1,\Phi)$ instead of~\eqref{eq:FullFlags}.

We would like to define some stability conditions on $\Higgs_{h,\delta}$. Consider two sequences of real numbers
\begin{equation*}
\begin{split}
    &0<a_1<\ldots<a_h<1\\
    &0<b_1<\ldots<b_h<1.
\end{split}
\end{equation*}
Let $c\in[0,1)$ be the unique number such that
\[
    \delta:=-c-\sum_ia_i-\sum_ib_i\in\Z.
\]

Let $(E,E_0^i,E_\infty^i,E_1^1,\Phi)\in\Higgs_{h,\delta}$. For a subbundle $E'\subset E$ define the set of jumps
\[
    I_0(E'):=\{i:E_0'\cap E_0^{h-i}\ne E_0'\cap E_0^{h-i+1}\}\subset\{1,\ldots,h\}.
\]
Define $I_\infty(E')$ similarly. Set
\[
\deg_{a,b}(E'):=
\begin{cases}
\deg E'+\sum_{i\in I_0(E')}a_i+\sum_{i\in I_\infty(E')}b_i\text{ if }E_1^1\not\subset E'_1\\
c+\deg E'+\sum_{i\in I_0(E')}a_i+\sum_{i\in I_\infty(E')}b_i\text{ if }E_1^1\subset E'_1.
\end{cases}
\]
We say that $(E,E_0^i,E_\infty^i,E_1^1,\Phi)$ is $(a,b)$-stable, if for every subbundle $E'\subset E$ such that $E'$ is preserved by $\Phi$ and $E'\ne0,E$, we have $\deg_{a,b}(E)<0$ (note that $\deg_{a,b}(E)=0$).

\begin{theorem}\label{th:Higgs}
Assume that for all pairs $i$ and $j$ we have $a_i+b_j\ne1$. Then there is a unique $(a,b)$-stable point $(E,E_0^i,E_\infty^i,E_1^1,\Phi)$ in $\Higgs_{h,\delta}$. Also, there is a decomposition $E=\bigoplus_p E^{(p)}$ such that

\stepzero\noindstep $\Phi(E^{(p)})\subset E^{(p-1)}\otimes\Omega_{\P_\C^1}(0+1+\infty)$;

\noindstep This decomposition is compatible with the flags in the sense that for $z\in\{0,1,\infty\}$ and all $i$ we have
\[
    E_z^i=\bigoplus_p\bigl(E_z^{(p)}\cap E_z^i\bigr).
\]

\noindstep $\rk E^{(p)}$ is given by the formula of Theorem~\ref{th:GolyshevCorti} with $\alpha_i=1-a_{h+1-i}$, $\beta_i=b_i$.
\end{theorem}

\begin{proof}
The data $(E,E_0^i,E_\infty^i,E_1^1,a_i,b_i)$ gives rise to a filtered vector bundle, that is, to a vector bundle
\[
    \underline E:=E|_{\P_\C^1-\{0,1,\infty\}}
\]
with a family $\underline E(\alpha)$ of extensions to $\P_\C^1$ parameterized by $\alpha\in\R$, see~\cite[Synopsis]{SimpsonHarmonicNoncompact}. Comparing our definitions with that of~\cite{SimpsonHarmonicNoncompact}, one sees that an $(a,b)$-stable point $(E,E_0^i,E_\infty^i,E_1^1,\Phi)$ gives rise to a stable filtered regular Higgs bundle of degree zero. Using~\cite[p.~718, Theorem]{SimpsonHarmonicNoncompact}, we see that this Higgs bundle corresponds to a stable filtered regular D-module over $\P_\C^1-\{0,1,\infty\}$. The local monodromy of this D-module is calculated using the table on page 720 of~\cite{SimpsonHarmonicNoncompact}, so we see that the corresponding bundle with connection is isomorphic to $(V,\nabla)$ from Section~\ref{sect:MonIrrRig} with  $\alpha_i=1-a_{h+1-i}$, $\beta_i=b_i$ (use rigidity). Moreover, the corresponding extensions of~$V$ to $\P_\C^1$ at each of the points $0$, $1$, and $\infty$ are the extensions $V^{>-\alpha}$ defined after Lemma~\ref{lm:DeligneExt}. Indeed, it follows from the same table that the jumps of the filtration are opposite to the residues of the connection (note that our convention for the residue of the connection is opposite to that of Simpson).

Let $H^\bullet$ be a CPVHS on $(V,\nabla)$ from Theorem~\ref{th:GolyshevCorti} and let $F^pV$ be the corresponding holomorphic filtration on $V$. Then $\underline E(\alpha)=\gr_FV^{>-\alpha}$, $\Phi=\gr_F\nabla$. Now one easily checks all the statements of the theorem with $E^{(p)}=\gr^p_FV^{>0}$ (cf.~\cite[Thm.~2.2.2]{DettweilerSabbah}).
\end{proof}

\begin{remark}
One can also consider the case when some of the $a_i$ or $b_i$ coincide. This corresponds to the case of partial flags at $z=0$ and $z=\infty$. We leave the formulation to the reader.
\end{remark}

\section{Middle convolution}\label{sect:MC}
In this section we recall the notion of the Katz middle convolution. Every rigid vector bundle with connection can be obtained from a rank one bundle with connection by iterated application of Katz middle convolution and tensoring with rank one bundle with connection (see~\cite[Sect.~5]{Simpson:MiddleConv} and also~\cite{KatzRigid,ArinkinRigid,DettweilerSabbah,DettweilerReiter,BlochEsnault:Fourier}). In Lemma~\ref{lm:MC} we show how this applies to the hypergeometric bundle with connection $(V,\nabla)$ considered above. This seems to be well-known but the author have been unable to find a reference.

We will use the language of D-modules, so let us give a dictionary (all D-modules are assumed algebraic below). Let $(V,\nabla)$ be a holomorphic vector bundle with connection on a subset $U$ of $\A^1_\C$ such that $\Sigma:=\A^1_\C-U$ is finite. By~\cite[Ch.~2, (1.5)]{MalgrangeDiffEqBook} we can uniquely extend $(V,\nabla)$ to an algebraic bundle with meromorphic connection on $U$ such that the connection has regular singularities at $\Sigma\cup\infty$. Denote this bundle with connection by $(V_{alg},\nabla_{alg})$. Then the minimal extension $M=\jmath_{!*}(V_{alg},\nabla_{alg})$, where $\jmath:U\to\A^1_\C$ is the inclusion, is a holonomic D-module on $\A^1_\C$ with regular singularities at $\Sigma\cup\infty$. This construction gives a bijection between irreducible holomorphic vector bundles with connection on $U$ and irreducible not supported at single point holonomic D-modules on $\A^1_\C$ whose singularities are regular and contained in $\Sigma\cup\infty$. We will often use this identification, sometimes implicitly. By a CPVHS on $M$ we mean a CPVHS on $(V,\nabla)$.

Let us recall the notion of middle convolution with Kummer D-modules. For $\alpha\notin\Z$, let $K_\alpha$ be the Kummer D-module $(\C[z,z^{-1}],\mathbf{d}-\alpha\frac{dz}z)$.
(We are using additive notation, so $K_\alpha$ is denoted by $K_{\E(\alpha)}$ in~\cite[Sect.~1.1]{DettweilerSabbah}.) Then, for a holonomic D-module $M$ on $\A^1_\C$, its middle convolution with $K_\alpha$ (denoted $MC_\alpha(M)$) is uniquely defined by the condition that its Fourier transform $^FMC_\alpha(M)$ is the minimal extension at the origin of $^FM\otimes L_{-\alpha}$, see~\cite[Prop.~1.1.8]{DettweilerSabbah}. Note that $MC_\alpha(M)$ is also a holonomic D-module on $\A^1_\C$. Again, we are using additive notation, so $MC_\alpha$ is denoted by $MC_{\E(\alpha)}$ in~\cite[(1.1.7)]{DettweilerSabbah}.

Now we want to explain how to obtain the hypergeometric bundle with connection $(V,\nabla)$ from a similar bundle with connection of smaller rank via a middle convolution and tensoring with line bundles. Since $z=\infty$ is a special point for the middle convolution, we need to move the singularity of $(V,\nabla)$ from $z=\infty$ to another point (cf.~\cite[Assumption 1.2.2(2)]{DettweilerSabbah}). Let $\phi:\P_\C^1\to\P_\C^1$ be the projective transformation sending $(0,1,2)$ to $(0,1,\infty)$. Then $\phi^*(V,\nabla)$ is a bundle with connection on $\P^1_\C-\{0,1,2\}$. Let us restrict it to $\A^1_\C-\{0,1,2\}$ and then extend to an irreducible holonomic D-module $M$ on $\A^1_\C$ with regular singularities at $z=0$, 1, 2, and $\infty$ as explained in the beginning of this section.

We see from Proposition~\ref{pr:localmonodromy} that the local monodromy of $M$ at $z=0$ is conjugate to $\E(A)$, its local monodromy at $z=2$ is conjugate to $\E(B)$, and its monodromy at $z=1$ is a quasi-reflection. The monodromy at $z=\infty$ is trivial.

Let the irreducible D-module $M_{k,j}$ on $\A_\C^1$ be constructed in the same way as $M$ from the differential operator
\[
    \prod_{m\ne k}(D-\alpha_m)-z\prod_{m\ne j}(D-\beta_m),\qquad D=z\frac{d}{dz}.
\]
Now Proposition~\ref{pr:localmonodromy} describes the local monodromy of $M_{k,j}$. In particular, the monodromies at $z=0$ and $z=2$ are regular operators, the monodromy at $z=1$ is a pseudo-reflection, and the monodromy at $z=\infty$ is trivial.

Let $L_{k,j}$ be (the D-module corresponding to) a line bundle with connection on $\P_\C^1-\{0,2,\infty\}$ with monodromy $\E(\alpha_k)$ at $z=0$, monodromy $\E(-\beta_j)$ at $z=2$ and monodromy $\E(\beta_j-\alpha_k)$ at $z=\infty$. Similarly, let $L_{k,j}'$ be a line bundle with connection on $\P_\C^1-\{0,2,\infty\}$ with monodromy $\E(-\beta_j)$ at $z=0$, monodromy $\E(\alpha_k)$ at $z=2$ and monodromy $\E(\beta_j-\alpha_k)$ at $z=\infty$.

\begin{lemma}\label{lm:MC}
For any $k,j\in\{1,\ldots,h\}$ we have
\begin{equation}\label{eq:MC}
    M\simeq\left(MC_{\beta_j-\alpha_k}(M_{k,j}\otimes L_{k,j}')\right)\otimes L_{k,j}.
\end{equation}
\end{lemma}
\begin{proof}
First of all, the generic rank of $MC_{\beta_j-\alpha_k}(M_{k,j}\otimes L_{k,j}')$ is
\begin{equation}\label{eq:MCrank}
    (h-1)+(h-1)+1-(h-1)=h
\end{equation}
by~\cite[(1.4.2)]{DettweilerSabbah}.

The only non-zero $\mu$-invariants of $M_{k,j}\otimes L_{k,j}'$ at $z=0$ are (cf.~\cite[(1.2.3) and (1.2.5)]{DettweilerSabbah}) given by
\[
\begin{split}
\mu_{\alpha_m-\beta_j,\mult(\alpha_m)-1}=1&\text{ if }\alpha_m\ne\alpha_k\\
\mu_{\alpha_k-\beta_j,\mult(\alpha_m)-2}=1&\text{ if }\mult(\alpha_k)\ge2.
\end{split}
\]
(We are using that $\alpha_m\ne\beta_j$.) The monodromy at $z=\infty$ is $\E(\beta_j-\alpha_k)\cdot\Id$.

By~\cite[Prop.~1.3.5]{DettweilerSabbah} the only non-zero $\mu$-invariants of $MC_{\beta_j-\alpha_k}(M_{k,j}\otimes L_{k,j}')$ at $z=0$ are
\[
\begin{split}
\mu_{\alpha_m-\alpha_k,\mult(\alpha_m)-1}=1&\text{ if }\alpha_m\ne\alpha_k\\
\mu_{0,\mult(\alpha_m)-2}=1&\text{ if }\mult(\alpha_k)\ge2.
\end{split}
\]
Since the rank of $MC_{\beta_j-\alpha_k}(M_{k,j}\otimes L_{k,j}')$ is $h$, one derives that the monodromy of $MC_{\beta_j-\alpha_k}(M_{k,j}\otimes L_{k,j}')$ at $z=0$ is a regular operator whose list of eigenvalues is $\E(\alpha_1-\alpha_k),\ldots,\E(\alpha_h-\alpha_k)$. Similarly, the monodromy of $MC_{\beta_j-\alpha_k}(M_{k,j}\otimes L_{k,j}')$ at $z=2$ is a regular operator whose list of eigenvalues is $\E(\beta_j-\beta_1),\ldots,\E(\beta_j-\beta_h)$.

Next, since the monodromy of $M_{k,j}\otimes L_{k,j}'$ at $z=1$ is a pseudo-reflection, there is a unique non-zero $\mu$-invariant $\mu_{\gamma,0}=1$ at $z=1$. By~\cite[Prop.~1.3.5]{DettweilerSabbah} $MC_{\beta_j-\alpha_k}(M_{k,j}\otimes L_{k,j}')$ also has a unique non-zero $\mu$-invariant $\mu_{\gamma',0}=1$  at $z=1$. It follows that the corresponding local monodromy is also a pseudo-reflection.

Finally, the monodromy of $MC_{\beta_j-\alpha_k}(M_{k,j}\otimes L_{k,j}')$ at $z=\infty$ is $\E(\alpha_k-\beta_j)\cdot\Id$; this follows from~\cite[Cor.~1.4.1]{DettweilerSabbah}.

We see that the left and the right hand sides of~\eqref{eq:MC} have the same local monodromy (see the remark after the proof of Proposition~\ref{pr:localmonodromy}). Since $M$ is physically rigid, the statement follows.
\end{proof}

\section{Behavior of local Hodge data under middle convolution}
Let $\alpha,\beta,\gamma\in[0,1)$ be distinct numbers. The notation $\alpha\to\beta\to\gamma$ will mean that $\E(\alpha)$, $\E(\beta)$ and $\E(\gamma)$ appear on the unit circle in the counterclockwise order.

Recall that we have fixed sequences $\alpha_i$ and $\beta_j$ satisfying~\eqref{eq:alphabeta}. \emph{In this section we assume that $\alpha_1>0$.} Fix $k,j\in\{1,\ldots,h\}$. Let $M_{k,j}$ be as in the previous section. Let us put a CPVHS on $M_{k,j}$ (existing by Proposition~\ref{pr:ExistsUniqueCPVHS}). Then, according to Lemma~\ref{lm:MC},~\cite[Prop.~3.1.1]{DettweilerSabbah}, and~\cite[Sect.~2.3]{DettweilerSabbah}, we get an induced CPVHS on~$M$. Let $\mu^p_{\alpha,l}$ denote the local Hodge invariants at $z=0$ for $M$, while ${^{k,j}\mu}^p_{\alpha,l}$ denote similar invariants for $M_{k,j}$.

\begin{proposition}\label{pr:muMC}
\stepzero\noindstep\label{lm:muMC1}
For all $\alpha_m\ne\alpha_k$, $l\ge0$ we have
\begin{equation*}
\mu^p_{\alpha_m,l}=
\begin{cases}
{^{k,j}\mu}^p_{\alpha_m,l}\text{ if $\alpha_m\to\alpha_k\to\beta_j$}\\
{^{k,j}\mu}^{p-1}_{\alpha_m,l}\text{ if $\alpha_k\to\alpha_m\to\beta_j$.}
\end{cases}
\end{equation*}

\noindstep\label{lm:muMC2}
For $l\ge1$ we have
\begin{equation*}
\mu^p_{\alpha_k,l}=
{^{k,j}\mu}^{p-1}_{\alpha_k,l-1}.
\end{equation*}
\end{proposition}

\begin{proof}
In the notation of Lemma~\ref{lm:MC}, let ${'\mu}_{\alpha,l}^p$ denote the local Hodge invariants at $z=0$ for $M\otimes L_{k,j}^{-1}$, while
${''\mu}_{\alpha,l}^p$ denote the local Hodge invariants at $z=0$ for $M_{k,j}\otimes L_{k,j}'$. Recall that we assumed that $\alpha_1>0$. Thus, according to~\cite[(2.2.13)]{DettweilerSabbah}, part~\eqref{lm:muMC1} of the proposition is equivalent to
\begin{equation}\label{eq:muMCtwisted}
{'\mu}^p_{\alpha_m-\alpha_k,l}=
\begin{cases}
{''\mu}^p_{\alpha_m-\beta_j,l}\text{ if $\alpha_m\to\alpha_k\to\beta_j$}\\
{''\mu}^{p-1}_{\alpha_m-\beta_j,l}\text{ if $\alpha_k\to\alpha_m\to\beta_j$.}
\end{cases}
\end{equation}
Similarly, according to~\cite[(2.2.14)]{DettweilerSabbah}, part~\eqref{lm:muMC2} is equivalent to
\begin{equation}\label{eq:muMCtwisted2}
{'\mu}^p_{0,l-1}={''\mu}^{p-1}_{\alpha_k-\beta_j,l-1}.
\end{equation}

To prove~\eqref{eq:muMCtwisted}, we may apply~\cite[Thm.~3.1.2(2)]{DettweilerSabbah} with $\alpha_0=\{\alpha_k-\beta_j\}$, $\alpha=\{\alpha_k-\alpha_m\}$.
We get
\begin{equation*}
{'\mu}^p_{\alpha_m-\alpha_k,l}=
\begin{cases}
{''\mu}^{p}_{\alpha_m-\beta_j,l}\text{ if }\{\alpha_k-\alpha_m\}<\{\alpha_k-\beta_j\}\\
{''\mu}^{p-1}_{\alpha_m-\beta_j,l}\text{ otherwise.}
\end{cases}
\end{equation*}
It remains to check that
$\{\alpha_k-\alpha_m\}<\{\alpha_k-\beta_j\}$ if and only if $\alpha_m\to\alpha_k\to\beta_j$.
Finally,~\eqref{eq:muMCtwisted2} follows from~\cite[Thm.~3.1.2(2)]{DettweilerSabbah} with $\alpha=0$.
\end{proof}

\section{Proof of Theorem~\ref{th:nu}}\label{sect:proof}
Let $(V,\nabla)$ the hypergeometric bundle with connection from the statement of the theorem. As explained before, it is more convenient for us to work with the bundle with connection $(\tilde V,\tilde\nabla):=\phi^*(V,\nabla)$ on $\P_\C^1-\{0,1,2\}$. Recall that $M$ denotes the D-module on $\A_\C^1$ corresponding to $(\tilde V,\tilde\nabla)$. Clearly, it is enough to prove the (analogue of the) theorem for a CPVHS on $(\tilde V,\tilde\nabla)$.

The theorem is obvious for $h=1$. Thus we may assume that $h\ge2$. We may also assume that the theorem is proved for all ranks smaller than $h$.

The statement of Theorem~\ref{th:nu} is invariant under subtracting the same number from all $\alpha_i$ and $\beta_i$ as explained below. Set $\alpha'_i=\{\alpha_i-\gamma\}$, $\beta'_i=\{\beta_i-\gamma\}$. Let $M'$ be the D-module constructed in the same way as $M$ from $\alpha'_i$ and $\beta'_i$. Then the D-module $M'$ is isomorphic to $M\otimes K'_\gamma$, where $K'_\gamma$ is a line bundle with connection on $\P_\C^1-\{0,2\}$ with monodromy $\E(-\gamma)$ around zero. Thus given a CPVHS on $M$, we get a CPVHS on $M'$.

\begin{lemma}\label{lm:addingSameNumber}
Assume that the local Hodge invariants at $z=0$ and $z=2$ for a CPVHS on $M$ are given by formulas of Theorem~\ref{th:nu}\eqref{thitem:0} and Theorem~\ref{th:nu}\eqref{thitem:2} respectively. Then the local Hodge invariants for the corresponding CPVHS on $M'$ are given by the same formulas with $\alpha_i$ and $\beta_i$ replaced by $\alpha'_i$ and $\beta'_i$ up to a shift.
\end{lemma}
\begin{proof}
First of all, the local Hodge invariants of $M$ and $M'$ are related via~\cite[(2.2.13)]{DettweilerSabbah}.

If the lemma holds for certain $\gamma$ and all sequences~\eqref{eq:alphabeta}, then it holds for any multiple of~$\gamma$. Thus we may assume that $\gamma$ is small in the sense that $\gamma<\min(\alpha_k,\beta_j)$, where~$\alpha_k$ is the smallest positive element of the set $\{\alpha_1,\ldots,\alpha_h\}$, $\beta_j$ is the smallest positive element of the set $\{\beta_1,\ldots,\beta_h\}$.

Assume that Theorem~\ref{th:nu} holds for some sequences $\alpha_i$ and~$\beta_i$. Consider the case $\alpha_1=0$ (then $\gamma<\beta_1$). For a number $\alpha$ we use notation $\alpha':=\{\alpha-\gamma\}$. Let $\alpha\in\{\alpha_1,\ldots,\alpha_h\}$. By \cite[(2.2.13)]{DettweilerSabbah} the only non-zero invariant of the form $\nu_{\alpha',l}^p$ for $M'$ corresponds to
\[
\begin{split}
    p=\#\{i:\alpha_i\le\alpha\}-\#\{i:\beta_i\le\alpha\}&=\#\{i:\alpha'_i\le\alpha'\}-\#\{i:\beta'_i\le\alpha'\}+\mult(\alpha_1)\\
    l=\mult(\alpha')&=\mult(\alpha).
\end{split}
\]
This coincides with the required formula up to a shift by $\mult(\alpha_1)$. One checks similarly, that the Hodge invariants at $z=2$ undergo a shift by $\mult(\alpha_1)$ as well. Now use Remark~\ref{rm:shift}.

The case $\beta_1=0$ is completely similar. The case $\alpha_1>0$, $\beta_1>0$ is easier and is left to the reader.
\end{proof}

Using the lemma, we may assume that $0<\alpha_1<\beta_1$ and that the $\mult(\alpha_1)$ is the largest among $\mult(\alpha_j)$ through the end of the section. In particular, we have an equality of local Hodge invariants at $z=0$
\[
    \nu_{\alpha_m,l}^p=\mu_{\alpha_m,l}^p
\]
for all $m$, $l$ and $p$ by~\eqref{eq:munu}. We start with the part~\eqref{thitem:0} of Theorem~\ref{th:nu}.

\begin{proposition}\label{pr:thm_a}
The local Hodge invariants at $z=0$ for a CPVHS on $M$ satisfy the formulas of Theorem~\ref{th:nu}\eqref{thitem:0} up to a~shift.
\end{proposition}
Our strategy is as follows. Since we are assuming that the theorem is true for all ranks smaller than $h$, the proposition is also true for all ranks smaller than $h$. First, we prove the proposition in the case of the low rank $h=2$ and $\alpha_1<\alpha_2$. Then we prove the proposition in the case when the rank is arbitrary but we have $\alpha_1<\alpha_2<\ldots<\alpha_h$. Finally, we prove the proposition in the general case.

If the eigenvalues are distinct, we can only have $\nu_{\alpha_m,l}^p$ and $\mu_{\alpha_m,l}^p$ non-zero, if $l=0$. Thus, we skip the index $l$ in the next two subsections.

\subsection{Rank two case---distinct eigenvalues}
According to Lemma~\ref{lm:MC}, we have
\[
    M\simeq MC_{\beta_2-\alpha_2}(M_{2,2}\otimes L_{2,2}')\otimes L_{2,2}.
\]
Using Lemma~\ref{lm:addingSameNumber}, we may assume that we are in the one of two cases
\begin{equation}\label{eq:cases}
\begin{split}
&0<\alpha_1<\alpha_2<\beta_1\le\beta_2\qquad\text{(Case I)}\\
&0<\alpha_1<\beta_1<\alpha_2<\beta_2\qquad\text{(Case II)}.
\end{split}
\end{equation}
We put on $M_{2,2}$ the unique CPVHS such that $h^1=1$, $h^p=0$ if $p\ne1$. For this CPVHS the only non-zero local Hodge invariant at $z=0$ is ${^{2,2}}\nu_{\alpha_1}^1=1$ (notation of Proposition~\ref{pr:muMC}). As before, this induces a CPVHS on $M$, and Proposition~\ref{pr:muMC} gives
\[
    \nu_{\alpha_1}^p=
    \begin{cases}
    1\text{ if }p=1\\
    0\text{ otherwise,}
    \end{cases}
\]
which agrees with formulas of Theorem~\ref{th:nu}\eqref{thitem:0}.

Next, we will calculate the Hodge numbers $h^p$ for the CPVHS on $M$. Note that $M_{2,2}\otimes L_{2,2}'$ is (a D-module corresponding to) a line bundle with connection on $\P_\C^1-\{0,1,2,\infty\}$. Let $(V^0,\nabla^0)$ be the extension of this line bundle with connection to a line bundle with a singular connection on $\P_\C^1$ such that the residues of the connection belong to $[0,1)$ (see Lemma~\ref{lm:DeligneExt}). It is easy to see that
\[
    (V^0,\nabla^0)=\left(\cO_{\P^1}(-\delta), \mathbf{d}+\frac{\{\beta_2-\alpha_1\}\,dz}z+\frac{\{\alpha_1-\beta_1\}\,dz}{z-1}+
    \frac{\{\beta_1-\alpha_2\}\,dz}{z-2}\right),
\]
where $\delta$ is chosen so that the residue at $z=\infty$ is $\{\alpha_2-\beta_2\}$ (recall that the degree of the line bundle is opposite to the sum of the residues of a connection).

In Case I of~\eqref{eq:cases}, we have
\[
    \delta=(\beta_2-\alpha_1)+(\alpha_1-\beta_1+1)+(\beta_1-\alpha_2)+(\alpha_2-\beta_2+1)=2.
\]
Similarly, in Case II we have $\delta=3$.

The CPVHS on $M_{2,2}$ induces on $M_{2,2}\otimes L_{2,2}'$ a CPVHS, whose invariants at $z=0$ we denote by ${'}h^p$, ${'}\nu_\alpha^p$ etc.
Clearly, ${'}h^1=1$, ${'}h^p=0$ if $p\ne1$.

By~\cite[Def.~2.3.1]{DettweilerSabbah} we have
\[
{'}\delta^1=-\delta=
\begin{cases}
-2\text{ in Case I}\\
-3\text{ in Case II}
\end{cases}
\]
The other ${'}\delta^p$ are zero.

Next, note that the local Hodge invariants ${'\nu}_\alpha^p$ for $M_{2,2}\otimes L_{2,2}'$ at $z=0$, 1, and~2 are zero unless $p=1$. Thus by~\cite[(2.3.5*)]{DettweilerSabbah} for the CPVHS on $M$ we have
\[
h^1=\delta-{'}h^1=
\begin{cases}
1\text{ in Case I}\\
2\text{ in Case II},
\end{cases}
\]
and $h^p=0$ for $p\ne1,2$. Since $\sum_ph^p=\rk(M)=2$, we get
\[
h^2=
\begin{cases}
1\text{ in Case I}\\
0\text{ in Case II}.
\end{cases}
\]
Finally, by~\cite[(2.2.2**)]{DettweilerSabbah}, we have $\sum_m\nu_{\alpha_m}^p=h^p$. Thus we get
\[
\begin{split}
&\nu_{\alpha_2}^2=1\text{ in Case I}\\
&\nu_{\alpha_2}^1=1\text{ in Case II},
\end{split}
\]
all other $\nu_{\alpha_2}^p$ being zero. This proves the proposition in the rank two case.

\subsection{Distinct eigenvalues}
Here we assume that $h\ge3$ and
\[
0<\alpha_1<\ldots<\alpha_h<1.\\
\]

We need to prove that up to a shift
\begin{equation}\label{eq:nudist}
\mu_{\alpha_m}^p=
\begin{cases}
    1\text{ if }p=m-\#\{i:\beta_i<\alpha_m\}\\
    0\text{ otherwise.}
\end{cases}
\end{equation}

Recall that we assumed that this is true for all smaller values of~$h$. Thus for all $k,j=1,\ldots,h$, we equip $M_{k,j}$ with a CPVHS such that for $m\ne k$ we have
\begin{equation}\label{eq:nuind}
{^{k,j}\mu}_{\alpha_m}^p=
\begin{cases}
    1\text{ if }p=\#\{i:\alpha_i<\alpha_m,i\ne k\}-\#\{i:\beta_i<\alpha_m,i\ne j\}\\
    0\text{ otherwise.}
\end{cases}
\end{equation}

Let us equip $M$ with a CPVHS. Using uniqueness of CPVHS up to a shift and Proposition~\ref{pr:muMC}, we see that there are constants $c_{k,j}\in\Z$ such that we have for all $k,j$ and $m\ne k$
\begin{equation}\label{eq:const}
\mu^p_{\alpha_m}=
\begin{cases}
{^{k,j}\mu}^{p+c_{k,j}}_{\alpha_m}\text{if $\alpha_m\to\alpha_k\to\beta_j$}\\
{^{k,j}\mu}^{p-1+c_{k,j}}_{\alpha_m}\text{ otherwise.}
\end{cases}
\end{equation}

Note that ``$+c_{k,j}$'' is necessary because the CPVHS is defined up to a shift. On the other hand, these constants are \emph{the same for all $m$ and $p$}.

\begin{lemma}
    The conditions~\eqref{eq:const} determine the numbers $\mu^p_{\alpha_m}$ and $c_{k,j}$ up to a simultaneous shift
    $c_{k,j}\mapsto c_{k,j}+c$, $\mu^p_{\alpha_m}\mapsto\mu^{p+c}_{\alpha_m}$.
\end{lemma}
\begin{proof}
Assume that we have $k,j,k',j'\in\{1,\ldots,h\}$. Since $h\ge3$, we may find $m$ such that $1\le m\le h$ and $m\ne k,k'$. Applying~\eqref{eq:const} first to $k,j,m$ and then to $k',j',m$, we see that $c_{k,j}-c_{k',j'}$ are uniquely defined by~\eqref{eq:const}. Thus the collection~$c_{k,j}$ is uniquely defined up to simultaneously adding a constant. Now the lemma is obvious.
\end{proof}

Next, let ${^*\mu}_{\alpha_m}^p$ be defined by formula~\eqref{eq:nudist}. In view of the previous lemma, it remains to prove
\begin{lemma}
The conditions \eqref{eq:const} are satisfied with ${^*\mu}_{\alpha_m}^p$ instead of $\mu_{\alpha_m}^p$ and
\[
    c_{k,j}=\begin{cases}
    0\text{ if }\alpha_k<\beta_j\\
    1\text{ if }\alpha_k>\beta_j.
    \end{cases}
\]
\end{lemma}
\begin{proof}
Consider six cases.

\stepzero\noindstep
$\alpha_m<\alpha_k<\beta_j$. In this case
\[
{^*\mu}_{\alpha_m}^p={^{k,j}\mu}_{\alpha_m}^p\qquad c_{k,j}=0
\]

\noindstep
$\alpha_m<\beta_j<\alpha_k$. In this case
\[
{^*\mu}_{\alpha_m}^p={^{k,j}\mu}_{\alpha_m}^p\qquad c_{k,j}=1
\]

\noindstep
$\alpha_k<\alpha_m<\beta_j$. In this case
\[
{^*\mu}_{\alpha_m}^p={^{k,j}\mu}_{\alpha_m}^{p-1}\qquad c_{k,j}=0
\]

\noindstep
$\alpha_k<\beta_j<\alpha_m$. In this case
\[
{^*\mu}_{\alpha_m}^p={^{k,j}\mu}_{\alpha_m}^p\qquad c_{k,j}=0
\]

\noindstep
$\beta_j<\alpha_k<\alpha_m$. In this case
\[
{^*\mu}_{\alpha_m}^p={^{k,j}\mu}_{\alpha_m}^p\qquad c_{k,j}=1
\]

\noindstep
$\beta_j<\alpha_m<\alpha_k$. In this case
\[
{^*\mu}_{\alpha_m}^p={^{k,j}\mu}_{\alpha_m}^{p+1}\qquad c_{k,j}=1
\]
We see that in each case~\eqref{eq:const} is satisfied.
\end{proof}

\subsection{End of proof of Proposition~\ref{pr:thm_a}}
Recall that we assumed that the proposition is proved for all ranks smaller than $h$. We also assumed that $\mult(\alpha_1)$ is largest among $\mult(\alpha_k)$, so we may assume that $\mult(\alpha_1)\ge2$. Recall also that $0<\alpha_1<\beta_1$. By Proposition~\ref{pr:muMC}, after shifting the filtration, we have for $\alpha_m\ne\alpha_1$
\[
\begin{cases}
\mu^p_{\alpha_m,l}={^{1,1}\mu^p_{\alpha_m,l}}\text{ if $\beta_1<\alpha_m$}\\
\mu^p_{\alpha_m,l}={^{1,1}\mu^{p-1}_{\alpha_m,l}}\text{ if $\beta_1>\alpha_m$}
\end{cases}
\]
and for $l\ge1$
\[
\mu^p_{\alpha_1,l}={^{1,1}\mu^{p-1}_{\alpha_1,l-1}}.
\]
Now, using the induction hypothesis, it is easy to prove the proposition for the local Hodge invariants at $z=0$ except for $\mu^p_{\alpha_1,0}$. However, we see that for some $p$ we have $\mu^p_{\alpha_1,\mult(\alpha_1)-1}=1$. Since we have (see~\cite[(1.2.3)]{DettweilerSabbah})
\[
    \mult(\alpha_1)=\mu_{\alpha_1}=\sum_l(l+1)\mu_{\alpha_1,l}=\sum_{p,l}(l+1)\mu^p_{\alpha_1,l},
\]
we see that $\mu^p_{\alpha_1,\mult(\alpha_1)-1}=1$ is the only non-zero invariant. It follows that ${\mu^p_{\alpha_1,0}=0}$ and Proposition~\ref{pr:thm_a} is proved.

\subsection{Hodge invariants at $z=2$}
Let $(V,\nabla)$ be the bundle with connection on $\P_\C^1-\{0,1,\infty\}$ constructed from sequences~\eqref{eq:alphabeta} as above and let $(\tilde V,\tilde \nabla):=\phi^*(V,\nabla)$ be the corresponding bundle with connection on $\P_\C^1-\{0,1,2\}$.

Let $\psi$ be a projective transformation taking $(0,1,2)$ to $(2,1,0)$. By rigidity, $\psi$ takes $(\tilde V,\tilde\nabla)$ to a similar bundle with connection $(\tilde V',\tilde\nabla')$ corresponding to $\alpha'_m=1-\beta_{h-m+1}$ and $\beta'_m=1-\alpha_{h-m+1}$ (we are using that $\alpha_1>0$, $\beta_1>0$). Then, according to Proposition~\ref{pr:thm_a}, $(\tilde V',\tilde\nabla')$ underlies a CPVHS whose local invariants at $z=0$ are given by Theorem~\ref{th:nu}\eqref{thitem:0} with $\alpha_m$, $\beta_m$ replaced by $\alpha'_m$, $\beta'_m$. This CPVHS on $(\tilde V',\tilde\nabla')$ induces a CPVHS $\tilde V=\bigoplus_p{'H}^p$ on $(\tilde V,\tilde\nabla)$ via a pullback by $\psi$.

We see that the non-zero local Hodge invariant  ${'\nu}_{-\beta_m,l}^p$ of $(\tilde V,\tilde\nabla)$ at $z=2$ with respect to the latter CPVHS correspond to $l=\mult(\beta_m)-1$ and
\begin{multline}\label{eq:star}
    p=\#\{i:1-\beta_i\le1-\beta_m\}-\#\{i:1-\alpha_i\le1-\beta_m\}=\\
        (h-\#\{i:\beta_i<\beta_m\})-(h-\#\{i:\alpha_i<\beta_m\})=\\
            \#\{i:\alpha_i<\beta_m\}-\#\{i:\beta_i<\beta_m\}.
\end{multline}
These are exactly the formulas of Theorem~\ref{th:nu}\eqref{thitem:2}.

Let $\tilde V=\bigoplus_pH^p$ be the CPVHS on $(\tilde V,\tilde\nabla)$ satisfying formulas of Theorem~\ref{th:nu}\eqref{thitem:0} (it exists by Proposition~\ref{pr:thm_a}). A priori, the Hodge decompositions $\tilde V=\bigoplus_pH^p$ and $\tilde V=\bigoplus_p{'H}^p$ can differ by a shift of gradation. We will see below that, in fact, they coincide, thus proving that the local Hodge invariants for $\tilde V=\bigoplus_pH^p$ at $z=2$ satisfy the formulas of Theorem~\ref{th:nu}\eqref{thitem:2}.

\subsection{End of proof of  Theorem~\ref{th:nu}}
According to~\cite[(2.2.2**) and (2.2.3)]{DettweilerSabbah} we have for the CPVHS $\tilde V=\bigoplus_pH^p$
\begin{equation}\label{eq:Lefsh0}
    h^p=\sum_{\alpha\in\{\alpha_1,\ldots,\alpha_h\}}\nu_\alpha^p=
        \sum_{\alpha\in\{\alpha_1,\ldots,\alpha_h\}}\sum_{l\ge0}\sum_{k=0}^l\nu_{\alpha,l}^{p+k}.
\end{equation}

In particular,
\begin{multline}\label{eq:max}
    \max\{p:h^p\ne0\}=\max\{p:\exists\alpha,l\text{ such that }\nu_{\alpha,l}^p\ne0\}=\\
    \max_m\bigl(\#\{i:\alpha_i\le\alpha_m\}-\#\{i:\beta_i\le\alpha_m\}\bigr).
\end{multline}
For $t\in[0,1)$, set
\[
    f(t)=\#\{i:\alpha_i\le t\}-\#\{i:\beta_i\le t\}.
\]
The graph of $f(t)$ is a union of horizontal intervals closed on the left and open on the right. Now it is easy to see that $f(t)$ attains its maximum on the union of intervals of the form $[\alpha_k,\beta_l)$; fix one of such intervals. From~\eqref{eq:max} we obtain
\[
    \max\{p:h^p\ne0\}=f(\alpha_k)=\max f(t).
\]

Similarly, for the CPVHS $\tilde V=\bigoplus_p{'H}^p$ we have

\begin{equation}\label{eq:Lefsh}
    {'h}^p=\sum_{\beta\in\{\beta_1,\ldots,\beta_h\}}{'\nu}_{-\beta}^p=
        \sum_{\beta\in\{\beta_1,\ldots,\beta_h\}}\sum_{l\ge0}\sum_{k=0}^l{'\nu}_{-\beta,l}^{p+k}.
\end{equation}
and we show as above using~\eqref{eq:star} that
\begin{multline*}
    \max\{p:{'h}^p\ne0\}=\max_m\bigl(\#\{i:\alpha_i<\beta_m\}-\#\{i:\beta_i<\beta_m\}\bigr)=\\ f(\beta_l)+1
    =\max f(t).
\end{multline*}
It follows that the two CPVHS are the same and we are done with part~\eqref{thitem:2} of the theorem.

Now consider $\beta\in\{\beta_1,\ldots,\beta_h\}$. Assume that $\mult(\beta)=j$ and $\beta=\beta_m=\beta_{m+1}=\ldots=\beta_{m+j-1}$ (then $\beta_{m-1}<\beta_m$ and $\beta_{m+j-1}<\beta_{m+j})$.
We have by Theorem~\ref{th:nu}\eqref{thitem:2}
\[
    \nu_{-\beta,l}^p=
    \begin{cases}
        1\text{ if }p=\rho(m)+1,\quad l=j-1\\
        0\text{ otherwise},
    \end{cases}
\]
where, as before, $\rho(m)=\#\{j:\alpha_j<\beta_m\}-m$.

Now we see from~\eqref{eq:Lefsh} that $\nu_{-\beta,l}^p$ contributes to $h^{\rho(m)+1}$, $h^{\rho(m)}$,\ldots, $h^{\rho(m)-j+2}$. Since $\rho(m+k)=\rho(m)-k$ for $0\le k\le j-1$, we see that $\beta_m=\ldots=\beta_{m+j-1}$ contributes to $h^{\rho(m)+1}$, $h^{\rho(m+1)+1}$, \ldots, $h^{\rho(m+j-1)+1}$. Therefore $h^p=\#\rho^{-1}(p-1)$. This completes the proof of Theorem~\ref{th:nu}.

\section{Proof of Theorem~\ref{th:real}}\label{sect:real}
Let $\cV$ be the local system of horizontal sections of $(V,\nabla)$.
\begin{lemma}
There is an anti-linear automorphism $\theta$ of $\cV$ with the property that $\theta^2=\Id_{\cV}$.
\end{lemma}
\begin{proof}
We claim that the complex conjugate local system $\bar\cV$ is isomorphic to $\cV$. Indeed, by Proposition~\ref{pr:localmonodromy} and~\eqref{eq:real} both local systems have the same local monodromy, so the statement follows from rigidity (Proposition~\ref{pr:rigidity}). Such an isomorphism gives an anti-linear automorphism $\theta$ of $\cV$.

Then $\theta^2$ is a scalar endomorphism of $\cV$ because $\cV$ is irreducible (Proposition~\ref{pr:irreducible}); write $\theta^2=\lambda\cdot\Id_{\cV}$. We claim that $\lambda\in\R_{>0}$. Indeed, choose any $x\in\P^1_\C-\{0,1,\infty\}$ and consider the restriction $\theta_x$ of $\theta$ to the fiber $\cV_x$ of $\cV$. Since the monodromy of~$\cV$ around $z=1$ is a pseudo-reflection, $\theta_x$ commutes with a pseudo-reflection and, thus, with a rank one operator. Let $v$ be any non-zero vector in the image of this rank one operator, then $\theta_xv=\xi v$ for $\xi\in\C$. We have $\lambda=\xi\bar\xi$ and the claim is proved. It remains to scale $\theta$ by $\sqrt\lambda$.
\end{proof}

Thus $\theta$ is a real structure on $\cV$, so we can write $\cV=\cV_\R\otimes\C$. This real structure induces a real structure on $V$.

Next, let $V=\bigoplus_p H^p$ be the Hodge decomposition from Theorem~\ref{th:GolyshevCorti}. Set ${'H}^p:=\overline{H^{-p}}$, then
$(V,\nabla,{'H}^\bullet)$ is another variation of Hodge structures on $(V,\nabla)$.

If $Q$ is a polarization of $(V,\nabla,H^\bullet)$, then $\bar Q$ given by $\bar Q(x,y):=\overline{Q(\bar x,\bar y)}$ is a polarization of $(V,\nabla,{'H}^\bullet)$. Thus $(V,\nabla,{'H}^\bullet)$ is a CPVHS so, by uniqueness, it differs from $(V,\nabla,H^\bullet)$ by a shift of gradation: there is $c\in\Z$ such that for all $p$ we have $H^p={'H}^{p-c}=\overline{H^{c-p}}$. Since the smallest $p$ such that $H^p\ne0$ is $p=p_-$, and the largest such $p$ is $p=p_+$, we see that $c=p_++p_-$.

For integers $p$ and $q$ such that $p+q=p_+-p_-$, set $H^{p,q}:=H^{p+p_-}$. Then
\[
    V=\bigoplus_{p+q=p_+-p_-}H^{p,q},
\]
and we get a real variation of Hodge structures $(\cV_\R,H^{\bullet,\bullet})$ of weight $k=p_+-p_-$.

Since $(V,\nabla)$ is irreducible, a horizontal non-degenerate hermitian form is unique up to a scaling by a real factor. Thus $\bar Q=\lambda Q$. Since $\overline{\bar Q}=Q$, we have $\lambda=\pm1$. Since $Q$ and $\bar Q$ are polarizations, we see that $\lambda=(-1)^k$, so $Q$ gives rise to a $(-1)^k$-symmetric form $Q_\R$ on $\cV_\R$. Now it is easy to see that $Q_\R$ or $-Q_\R$ is a real polarization; Theorem~\ref{th:real} is proved.

\subsection*{Acknowledgements} The author is grateful to Dima Arinkin, Mikhail Mazin, Anton Mellit, Vladimir Fock, Vladimir Rubtsov, and Masha Vlasenko for the interest to the subject and discussions. He also wants to thank Michael Dettweiler for bringing to author's attention~\cite{TerasomaRadon,TerasomaHodgeAndTate} and Claude Sabbah for useful remarks on an early draft.
The author is partially supported by NSF grant DMS-1406532.

Special thanks go to Vasily Golyshev who introduced the author to the problem. A part of this work was done, while the author was a member of Max Planck Institute in Bonn.

\bibliographystyle{alphanum}
\bibliography{../../RF}
\end{document}